\documentclass[reqno,12pt]{amsart}
\usepackage{amsmath}
\usepackage{amssymb}
\usepackage{amscd}
\usepackage{graphicx}
\usepackage{subfigure}
\newtheorem{theorem}{Theorem}
\newtheorem*{theorem*}{Theorem}
\newtheorem{lemma}[theorem]{Lemma}
\newtheorem*{lemma*}{Lemma}
\newtheorem{corollary}[theorem]{Corollary}

\newtheorem*{fact*}{Fact}

\newtheorem{rremark}{Remark}

\def\R{\mathbb R}
\def\C{\mathbb C}

\def\w{\omega}

\begin{document}

\title[Stability of travelling waves]{Stability of travelling wave solutions to the sine-{G}ordon equation }
\author{ C.K.R.T.\ Jones \& R. Marangell}
\begin{abstract} We give a geometric proof of spectral stability of travelling kink wave solutions to the sine-Gordon equation. For a travelling kink wave solution of speed $c \neq \pm 1$, the wave is spectrally stable. The proof uses the Maslov index as a means for determining the lack of real eigenvalues. Ricatti equations and further geometric considerations are also used in establishing stability. 
\end{abstract}

\maketitle{}
\section{Introduction }

The sine-Gordon equation: 
\begin{equation}\label{eq:gov}
u_{tt} = u_{xx} + \sin u
\end{equation}
\noindent has applications in many areas of physics and mathematics. It can be used to model magnetic flux propagation in long Josephson junctions: two ideal superconductors separated by a thin insulating layer \cite{asfcsr76}, \cite{derksetal03}. It can be thought of a model for mechanical vibrations of the so-called `ribbon pendulum' - the continuum limit of a line of pendula each coupled to their nearest neighbor via Hooke's law \cite{rbpm08}.  In biology, it has found applications in modeling the transcription and denaturation in DNA molecules \cite{msalerno91}. Further, it can be used to model propagation of a crystal dislocation, Bloch wall motion of magnetic crystals, propagation of ``splay-wave'' along a lipid membrane, and pseudo-spherical surfaces to name a few others (see \cite{asfcdm73} and the references therein).

In this paper, we consider solutions  of the form $v(x+ct,t)$ where $c$ is the (positive) speed of the 
traveling wave. Making the change of variable $z = x+ ct$ and substituting 
into (\ref{eq:gov}) gives
\begin{equation} \label{eq:trav}
c^2 v_{zz} + 2c v_{zt} + v_{tt} = v_{zz} + \sin v.
\end{equation}
A {\em travelling wave} solution will be a $t$ independent solution $v(z)$ to 
(\ref{eq:trav}). Thus it solves the (nonlinear) pendulum equation: 
\begin{equation}\label{eq:pend}
(c^2-1)v_{zz} = \sin v.
\end{equation}
A {\em kink wave}  solution is a travelling wave solution to (\ref{eq:trav}) corresponding to a heteroclinic orbit in the phase plane of (\ref{eq:pend}). 

Stability of a singularly perturbed kink wave solution was shown in \cite{derksetal03}. In \cite{rbpm08}, the Cauchy problem for the sine-Gordon equation in laboratory coordinates was studied using inverse scattering techniques. In this paper we take a more geometric approach determining the spectral stability of kink-wave solutions to equation (\ref{eq:gov}) via geometric considerations and elementary methods of ODE theory. 

\subsection{The Maslov index}

Note: In this paper we will use the description of the Maslov index as in \cite{rs93}. 

Let the matrix $J = \begin{pmatrix} 0 & -1 \\ 1 & 0 \end{pmatrix}$ denote the standard symplectic structure on $\R^2$. A line $\ell$ passing through the origin is 
considered Lagrangian in the sense that for any $v_1, v_2 \in \ell$ the inner product of $v_1$ with $J v_2$, $<v_1, J v_2> = 0$. Let $\ell(t) $ be a curve of lines in $\R^2$. If $\ell(t)$ can be written as the span of the $2 \times 1$ matrix $\begin{pmatrix} x(t) \\ y(t) \end{pmatrix}$, we will call the functions $x(t)$ and $y(t)$ a {\em frame} for $\ell(t)$. Alternatively we can view $\ell(t)$ as a curve in $\R P^1 \approx S^1$, and if $\begin{pmatrix} x(t) \\ y(t) \end{pmatrix}$ is a frame of $\ell$, then $\ell(t) = [x(t):y(t)]$.  Now let ${\bf a} = [a_1: a_2]$ be a fixed line in $\R^2$. Suppose that $\ell(t): [t_0, t_n] \to \R P^1$, and $\ell(t) = {\bf a}$ at $t_1, t_2, \ldots t_{n-1}$ 
with $t_0 < t_1 < t_2 < \cdots < t_{n-1} < t_n$. Suppose further that we have $\dot{\ell}(t_i) \neq (0,0)$ for all $t_i$. We define the {\em crossing form}, $\Gamma(\ell(t), {\bf a},t_i)$ of 
$\ell(t)$, with respect to ${\bf a}$, at $t_i$ as:
\begin{equation}\label{eq:cross}
\Gamma(\ell(t), {\bf a},t_i) = x(t_i)\dot{y}(t_i) - y(t_i)\dot{x}(t_i).
\end{equation}
\noindent For a curve $\ell(t)$ in $\R P^1$ as above, we define the {\em Maslov index}, $\mu(\ell(t), {\bf a})$, as: 
\begin{equation}\label{eq:mas}
\mu(\ell(t),{ \bf a}) = \sum_{t_i} \textrm{sign }(\Gamma(\ell(t), {\bf a}).
\end{equation}
The Maslov index defined in this way is a signed count of the number of times that $\ell(t) $, viewed as a curve in $S^1$ crosses the point {\bf a}.

\section{Kink Waves}
Travelling wave solutions to the sine-Gordon equation for which the quantity $c^2 -1 <0$ are called {\em subluminal} waves. When $c^2-1>0$ they are called {\em superluminal} waves. We have the following theorem: 

\begin{theorem} \label{th:main} Kink wave solutions to equation (\ref{eq:gov}) $u_{tt} = u_{xx} +\sin u$, are spectrally stable if $c^2\neq 1$.
\end{theorem} 

\subsection{Subluminal kink waves}
Until otherwise specified, the quantity $c^2-1<0$. We will focus primarily on the orbit that satisfies the following boundary conditions, though the analysis that follows will apply to all subluminal kink waves.
Let $v(z)$ be a solution to (\ref{eq:pend}) satisfying:
\begin{equation}
\lim_{z \to -\infty} v(z) = - \pi \textrm{ and } \lim_{z \to \infty} v(z) = \pi 
\end{equation}

Linearizing equation (\ref{eq:trav}) about the kink wave solution $v$ gives: 
\begin{equation}\label{eq:lin1}
(c^2-1)\varphi_{zz} + 2c\varphi_{zt} + \varphi_{tt} = (\cos v )\varphi.
\end{equation}
 \noindent Setting $\varphi_t = \psi$ we can rewrite (\ref{eq:lin1}) as:
 \begin{equation}\label{eq:lin2}
 \begin{pmatrix} \varphi \\ \psi \end{pmatrix}_t = 
 \begin{matrix}   \psi \\ \cos v \varphi - (c^2-1)\varphi_{zz} - 2c\psi_z \end{matrix} 
 \end{equation}
 \noindent In order to establish spectral stability, we need to consider the eigenvalue problem of equation (\ref{eq:lin2}). Letting $\lambda$ be the eigenvalue parameter, we obtain:
 \begin{equation}\label{eq:lin3} 
 \begin{matrix}   \psi \\ (\cos v) \varphi - (c^2-1)\varphi_{zz} - 2c\psi_z \end{matrix}  = \lambda \begin{pmatrix} \varphi \\ \psi \end{pmatrix} 
 \end{equation}
 We will consider (twice differentiable) $L^2$ perturbations, $\varphi$. This leads to the following eigenvalue condition: A function $\varphi$ is an eigenfunction of the linearized operator with eigenvalue $\lambda$ if 
 \begin{equation}\label{eq:mainode}
 (c^2-1) \varphi_{zz} + 2c\lambda \varphi_z + (\lambda^2 - \cos v) \varphi = 0,
 \end{equation} 
 is satisfied, together with the condition that 
 $$ \lim_{z \to \pm \infty} \varphi(z) = 0$$

The idea now is to reformulate the above eigenvalue condition in an equivalent, geometric way. Setting ${\bf w } = (w_1, w_2)$, with $w_1 = \varphi$, $w_2 = \varphi_z$ and $' = \frac{d}{dz}$, we have:
\begin{equation} \label{eq:odesys} {\bf w}' = 
\begin{pmatrix} w_ 1\\ w_2 \end{pmatrix}' = \begin{pmatrix} 0 & 1 \\ \frac{\cos v - \lambda^2}{c^2-1} & \frac{-2 c \lambda}{c^2-1} \end{pmatrix} \begin{pmatrix} w_1 \\ w_2 
\end{pmatrix} =: A(\lambda, z) {\bf w }. 
\end{equation}
\noindent Also we set 
\begin{equation}\label{eq:lima}
A(\lambda) := \lim_{z\to \pm \infty} A(\lambda,z) = \begin{pmatrix} 0 & 1 \\ \frac{- 1 - \lambda^2}{c^2-1} & \frac{-2 c \lambda}{c^2-1} \end{pmatrix}.
\end{equation}

We remark that $A(\lambda)$ has an unstable and a stable subspace which we will denote by $\xi^u$ and $\xi^s$ respectively.

\begin{lemma}[Geometric version of the eigenvalue condition]\label{lem:geom}
Given that $\displaystyle \begin{pmatrix} w_1 \\ w_2 \end{pmatrix}$ solves (\ref{eq:odesys})
\begin{equation}
\lim_{z\to \pm \infty} p_1 = 0 
\end{equation}
is equivalent to 
 \begin{equation}
\lim_{z\to -\infty} \begin{pmatrix} w_1 \\ w_2 \end{pmatrix} 
\to \xi^u  \textrm{, and}   \lim_{z \to \infty} \begin{pmatrix} w_1 \\ w_2 \end{pmatrix} \to \xi^s 
\end{equation}
\end{lemma}
\begin{proof}
 We compactify the extended phase plane of (\ref{eq:odesys}) by introducing a new variable $z = z(\tau) = \tan(\frac{\pi \tau}{2})$, where $\tau \in (-1,1)$. Now the extended system of ODE's becomes the (autonomous)
 \begin{equation}\label{eq:extendedode}
  \begin{array}{ccc}{\bf w'} & =& A(\tau,\lambda) {\bf w} \\
 \tau' & =  &\frac{2}{\pi}\cos^2{\frac{\pi \tau}{2}} \end{array}
  \end{equation} 
 Next we note that if we view the space $\R^2 \times (-1,1)$ as a plane bundle over the segment $[-1,1]$, the linearity of (\ref{eq:odesys}) means that linear subspaces of the fibres are preserved, so the flow defined by (\ref{eq:extendedode}) on $\R^2 \times (-1,1)$ induces a well defined flow on the cylinder $S^1 \times (-1,1)$. Moreover, this flow can be continuously extended to a flow on $S^1 \times [-1,1]. $ Finally we note that on this cylinder, the points $\xi^u$ and $\xi^s$ are fixed points of the induced flow, with a one dimensional unstable and stable manifold respectively. The uniqueness of the stable and unstable manifolds concludes the proof of the lemma. 
 \end{proof}

We have redefined the original eigenvalue condition as the existence of a heteroclinic orbit joining $\xi^u$ and $\xi^s$.

Now, as $c^2 -1<0$, the matrix $$A(\lambda, z) \to A(\lambda) = \begin{pmatrix} 0 & 1 \\ \frac{-1 - \lambda^2}{c^2-1} & \frac{-2c\lambda}{c^2-1} \end{pmatrix} $$ has eigenvalues $\gamma_u$ and $\gamma_s$ corresponding to the unstable and stable subspaces.
\begin{equation} \label{eq:eigenvalueA}
\gamma_u = \frac{-c\lambda - \sqrt{\lambda^2 - (c^2-1)}}{c^2-1} \textrm{ , and } \gamma_s = \frac{-c\lambda + \sqrt{\lambda^2 - (c^2-1)}}{c^2-1}.
\end{equation}
\noindent and $(c^2-1)\gamma_{u,s}^2 +2c\lambda \gamma_{u,s} +(\lambda^2+1) = 0$. Observe:
\begin{equation}\label{eq:eigenvectorsA}
\xi^u = < \begin{pmatrix} \frac{-c\lambda + \sqrt{\lambda^2 - (c^2-1)} }{\lambda^2 + 1} \\ 1 \end{pmatrix} >\textrm{ , and } 
\xi^s = < \begin{pmatrix} \frac{-c\lambda - \sqrt{\lambda^2 - (c^2-1)}}{\lambda^2 + 1} \\ 1 \end{pmatrix}>.
\end{equation}
\noindent where $< \begin{pmatrix} a \\ b \end{pmatrix} >$ denotes the linear space spanned by the vector $ \begin{pmatrix} a \\ b \end{pmatrix}$.  For each fixed $\lambda$, let $\ell(z)$ be the set of lines in $\R^2$ that tend to the unstable subspace of $A$ at $- \infty$, under the flow of (\ref{eq:odesys}), that is:
\begin{equation}\label{eq:unstab}
\ell(z) = \bigg{\{} <\begin{pmatrix} w_1 \\ w_2 \end{pmatrix} > | \begin{pmatrix} w_1 \\ w_2 \end{pmatrix} \textrm{ solves (\ref{eq:mainode}) and } \to \xi^u , \textrm{ as } z 
\to - \infty \bigg{\}}
\end{equation}
\noindent
The lines $\ell(z)$ are Lagrangian, and in fact the representation $\begin{pmatrix} w_1 \\ w_2 \end{pmatrix}$ is a frame for the line $\ell$, so we can define the
crossing form relative to the subspace $\xi^s$, $\Gamma(\ell(z), \xi^s,  \lambda, z)$ as $\Gamma(\ell(z), \xi^s,  \lambda, z) =   w_1w_2' - w_2w_1'$. Substituting as 
in equation (\ref{eq:odesys}) gives 
\begin{equation}\label{eq:cross1}
\Gamma(\ell(z), \xi^s,  \lambda, z) = \frac{\cos v - \lambda^2}{c^2-1} w_1^2 - \frac{2 c \lambda}{c^2-1} w_1 w_2 - w_2^2
\end{equation}
\begin{lemma}\label{lem:main} The crossing form of $\ell(z)$, relative to the stable subspace at infinity, $\Gamma(\ell(z), \xi^s,  \lambda, z)$ defined above is independent of $\lambda$. 
\end{lemma}
\begin{proof}
To evaluate the crossing form on $\xi^s$, we use the fact that on $ \xi^s$ we have that $w_2 = \frac{\lambda^2 + 1} {-c\lambda + \sqrt{\lambda^2 - (c^2-1)} }w_1 = \frac{(\lambda^2+1)}{\gamma_s(c^2-1)} w_1$. This gives 
\begin{eqnarray*}
\Gamma & = & \w_1^2 \Bigg{[} \frac{\cos v - \lambda^2}{c^2-1} - \frac{2 c \lambda}{c^2-1} \frac{(\lambda^2+1)}{\gamma_s(c^2-1)} - \bigg{(}\frac{(\lambda^2+1)}{\gamma_s(c^2-1)}\bigg{)}^2 \Bigg{]} \\
& = & \frac{\w_1^2}{(c^2-1)} \Bigg{[}\frac{(\cos v - \lambda^2)\gamma_s^2(c^2-1) - 2 c \lambda (\lambda^2 +1) \gamma_s - (\lambda^2+1)^2}{\gamma_s^2(c^2-1)}\Bigg{]} \\
& = & \frac{\w_1^2}{(c^2-1)} \Bigg{[}\frac{\cos v\gamma_s^2(c^2-1) - (\lambda^2)\gamma_s^2(c^2-1) - 2 c \lambda (\lambda^2 +1) \gamma_s - (\lambda^2+1)^2}{\gamma_s^2(c^2-1)}\Bigg{]} \\
& = & \frac{\w_1^2}{(c^2-1)} \Bigg{[}\frac{\cos v\gamma_s^2(c^2-1) - (\lambda^2)\gamma_s^2(c^2-1) - 2 c \lambda (\lambda^2 +1) \gamma_s - (\lambda^2+1)^2}{\gamma_s^2(c^2-1)}\Bigg{]} \\
& & + \frac{\w_1^2}{(c^2-1)} \Bigg{[} \frac{-\gamma_s^2(c^2-1) + \gamma_s^2(c^2-1)}{\gamma_s^2(c^2-1)}\Bigg{]} \\ 
& = &\frac{\w_1^2}{(c^2-1)} \Bigg{[}\frac{(\cos v +1)\gamma_s^2(c^2-1)-(\lambda^2+1)\big{(}(c^2-1)\gamma_s^2 +2c\lambda \gamma_s + (\lambda^2+1)\big{)}}{\gamma_s^2(c^2-1)}\Bigg{]} \\
& = &\frac{\w_1^2}{(c^2-1)} \Bigg{[}\frac{(\cos v +1)\gamma_s^2(c^2-1)}{\gamma_s^2(c^2-1)}\Bigg{]} 
\end{eqnarray*}
\noindent because $(c^2-1)\gamma_s^2 - 2 c \lambda \gamma_s + (\lambda^2 +1) = 0$, so we have that 
\begin{equation} \label{eq:gamind}
\Gamma(\ell(z), \xi^s,  \lambda, z) = \frac{(\cos v +1)\w_1^2}{(c^2-1)} 
\end{equation}
\noindent Which shows that $\Gamma$ is independant of $\lambda$.
\end{proof} 

Evaluating the limits and using lemma \ref{lem:main} we get that the Maslov index of the curve of Lagrangian subspaces $\ell(z)$, with respect to the stable subspace at infinty is: 
\begin{align}\label{eq:mas2}
  \mu(\ell(z), \xi^s, \lambda) & = \sum_{z} \textrm{ sign } \Gamma(\ell(z), \xi^s, \lambda, z) \\
  & = -  \# \textrm{ of crossings that occur. }
\end{align}

We next observe that only regular crossings occur. A crossing is called {\em regular} if $\dot{\ell(z)}$ is non-zero. Equation (\ref{eq:gamind}) is clearly 
non-zero, except at $z = \pm \infty$. 

At a regular crossing we have that $\ell(z,\lambda) = \begin{pmatrix} w_1 \\ w_2 \end{pmatrix} = \xi^s$. The function $f(z, \lambda) = \frac{w_2}{w_1}$, is well defined at a regular crossing of the unstable manifold $\ell(z)$ with $\xi^s$. Moreover  at a regular crossing we have that 
\begin{equation}\label{eq:ift1}
f(z,\lambda) =  \frac{\lambda^2 + 1}{-c\lambda + \sqrt{\lambda^2 - (c^2-1)} },
 \end{equation}
\noindent and that
\begin{equation}\label{eq:ift2}
\frac{\partial}{\partial z} f(z,\lambda) =  \frac{\Gamma(\ell(z), \xi^s, \lambda, z)}{w_1^2} \neq 0. 
\end{equation}
\noindent Thus we can use the implicit function theorem to write the location of the crossing in the $z$ variable as a function of the parameter $\lambda$, $z = z(\lambda)$, 
further because equation (\ref{eq:gamind}) is independant of $\lambda$, we can do this for all $\lambda$.

When $\lambda \gg 1 $, the system in (\ref{eq:odesys}) tends to 
\begin{equation}\label{eq:biglam}
\begin{pmatrix} w_1 \\ w_2 \end{pmatrix}' = \begin{pmatrix} 0 & 1 \\ \frac{- \lambda^2}{c^2-1} & \frac{-2c\lambda}{c^2-1} \end{pmatrix} \begin{pmatrix} w_1 \\w_2.\end{pmatrix} =: A(\lambda \gg 1) \begin{pmatrix} w_1 \\ w_2 \end{pmatrix}
\end{equation}
\noindent which has eigenvalues $\gamma_\pm = \frac{-\lambda}{c\pm1}$. Which are real and as $c^2-1<0$,one is positive and one is negative.  For large enough $\lambda$ then, the number of 
crossings of $\ell(z)$ relative to $\xi^s$ is the same as the number of crossings as for the solution of the constant coefficient equation (\ref{eq:biglam}) tending towards the unstable 
subspace of the matrix $A(\lambda \gg 1)$. Thus for large enough $\lambda$ the number of crossings of the unstable manifold $\ell(z)$ and $\xi^s$ is equal to zero. 

Now fix a $\lambda$ and a location of a crossing $z_1(\lambda)$ say. Because the number of crossings tends to zero as $\lambda$ increases, and because of the
implicit function theorem assertion above, we have that for some finite $\lambda_1$, $\displaystyle \lim_{\lambda \to \lambda_1} z_1(\lambda)  = \infty$. But this is exactly the geometric reformulation of the eigenvalue condition. Moreover, the implicit function theorem implies that we have a unique value, $z_1(\lambda)$,  for this crossing for each $\lambda$, and the locations of different crossings (say $z_2(\lambda)$) cannot intersect. Thus as $\lambda$ increases, the number of crossings will decrease monotonically, as will the Maslov index. This last paragraph is summed up in the following corollary.

\begin{corollary} \label{cor:main} If $N_1$ is the number of times $\ell(z)$ crosses $\xi^s$ when $\lambda = \lambda_1$, as $z$ ranges from $-\infty $ to $\infty$ and $N_2$ is the number of times $\ell(z)$ crosses $\xi^s$ when $\lambda = \lambda_2$, then $|N_1-N_2| = $ the number of eigenvalues of the linearized operator (\ref{eq:lin2}) in $(\lambda_1,\lambda_2)$.
\end{corollary}

Lastly, we show that $\mu(\ell(z), \xi^s, 0, z) = 0$. This shows that there are no crossings when $\lambda = 0$ and so there are no crossings for $\lambda \in (0, \infty).$ To see this observe that when $\lambda = 0$, equation (\ref{eq:odesys}) reduces to the equation of variations of the pendulum equation (\ref{eq:pend}). Thus the unstable manifold of equation (\ref{eq:odesys}) is just the tangent line to the heteroclinic orbit, joining $(-\pi,0)$ to $(\pi,0)$ in the phase plane of (\ref{eq:odesys}). This orbit in the phase plane is given by:
\begin{equation}\label{eq:hetero}
v_z =  \sqrt{2\frac{\cos v+1}{1-c^2}} 
\end{equation}
\noindent and the slope of the tangent line is given by 
\begin{equation}\label{eq:slope}
\frac{d v_z}{dv} = \frac{-\sin v}{\sqrt{(c^2-1)(2\cos v +2)}}.
\end{equation}
\noindent In order for a crossing to occur, the tangent line must be parallel to $\xi^s$. This means that
\begin{equation}
 \frac{-\sin v}{\sqrt{(c^2-1)(2\cos v +2)}} = \frac{-1}{\sqrt{-(c^2-1)}}
\end{equation}
\noindent But this only happens when $v = \pm \pi$. But $v = \pm \pi$ is a critical point in the phase plane. That is, there is no $z \in (-\infty, \infty)$ where this can happen, so there can be no crossings, so $\mu(\ell(z), \xi^s, 0, z) = 0$. This shows that there are no crossings for $\lambda \in (0, \infty)$. 

Lastly we note that the crossing form argument works in the negative $\lambda$ direction and that the asymptotic behaviour of the system in (\ref{eq:odesys}) for $\lambda \ll -1$ is the same as that for $\lambda \gg 1$. Thus we have no real eigenvalues $\lambda \neq 0$. 

\subsection{Complex Eigenvalues}
We now return our attention to equation (\ref{eq:mainode}) to show that there are no complex eigenvalues $\lambda = p+ i q$ where $p>0$. Without loss of generality we can assume that $q >0$ so $\lambda = r e^{i\theta}$ where $r>0$ and $ \theta \in (0, \frac{\pi}{2})$. 
Rewriting (\ref{eq:mainode}) as: 
\begin{equation}\label{eq:odemain}
w'' + \frac{2c\lambda}{c^2-1} w' + \frac{(\lambda^2-\cos(v(z)))}{c^2-1}w = 0, 
\end{equation}
\noindent we make the substitution 
\begin{equation}\label{eq:trick}
\psi = w e^{\big{(}\frac{c\lambda}{c^2-1} z\big{)}}
\end{equation}
\noindent and re-write (\ref{eq:odemain}) as 
\begin{equation}\label{eq:odesub}
\psi'' = \displaystyle \bigg{(} \frac{\cos v - \lambda^2}{c^2-1} + \frac{c^2\lambda^2}{(c^2-1)^2} \bigg{)} \psi
\end{equation}
We wish to consider now complex valued solutions to (\ref{eq:odesub}). We let $\eta = \frac{\psi'}{\psi}$ and consider the induced flow on a chart of $\C P^1$, where $\psi  \neq 0$:
\begin{equation}\label{eq:odecp1}
\eta ' = \frac{\psi''\psi - \psi' \psi'}{\psi^2} = \frac{\cos v - \lambda^2}{c^2-1} + \frac{c^2\lambda^2}{(c^2-1)^2} - \eta^2.
\end{equation}
\noindent If we set $\eta = \alpha + i\beta$, and consider the imaginary part of the vector field of (\ref{eq:odecp1}) on the real axis, that is those points where $\eta = \alpha$, we have 
\begin{equation}\label{eq:odeim}
\beta' \big{|}_{\beta = 0} = \frac{-2pq}{c^2-1} + \frac{2c^2pq}{(c^2-1)^2} = \frac{2pq}{(c^2-1)^2} > 0 \hspace{1in} \textrm{ if } p, q >0. 
\end{equation}
\noindent thus the flow of (\ref{eq:odecp1}) on the real axis of $\C P^1$ is always pointing in the positive imaginary direction. 

By considering $\frac{w_2}{w_1}$ from equation (\ref{eq:odesys}), we now interpret the eigenvalue conditions, for the eigenvalues that we are interested in.
\begin{lemma}\label{lem:eigenvalue}
An eigenvalue $\lambda$ is a value of $\lambda$ where there exists a heteroclinic orbit of the flow induced by (\ref{eq:odesys}), on (a 
chart of) $\C P^1$, i. e. $\lambda$ an eigenvalue means that there is an orbit from
$\displaystyle
\lim_{z \to -\infty} \frac{w_2}{w_1} = \gamma_u$ to $\displaystyle \gamma_s = \lim_{z \to \infty} \frac{w_2}{w_1} $. Under the transformation given in (\ref{eq:trick}) we have that 
\begin{equation}\label{eq:liouville-ricatti}
\eta = \frac{w_2}{w_1} + \frac{c\lambda}{c^2-1},\end{equation} 
and so $\gamma_u$ goes to $\eta_u$ and $\gamma_s$ goes to $\eta_s$ where 
\begin{eqnarray}
\eta_u & = & \frac{-\sqrt{\lambda^2 - (c^2-1)}}{c^2-1} \label{eq:etau} \textrm{ , and}\\
\eta_s & = & \frac{\sqrt{\lambda^2 - (c^2-1)}}{c^2-1} \label{eq:etas}.
\end{eqnarray}
\end{lemma}

We claim that for $\lambda = p + iq$ with $p$ and $q$ both positive, that $\eta_u$ has a positive imaginary part and $\eta_s$ has a negative imaginary part. 
Thus as $\beta'\big{|}_{\beta = 0} >0$, from above, there can be no orbit joining the two critical points, and so no eigenvalues with positve real part. Thus all kink waves 
satisfying $c^2-1 <0$ are stable.  

To see that $\eta_u$ has a positive imaginary part, write $\lambda = r e^{i\theta}$ and $\eta_u = R_u e^{i \theta_u}$. Since we are assuming that $\lambda$ has positive imaginary part $\theta \in (0,\frac{\pi}{2})$. By equation (\ref{eq:etau}) we have that 
\begin{equation}\label{eq:thetau}
\theta_u = \arctan\bigg{(}\frac{r^2 \sin(2 \theta)}{r^2 \cos(2\theta) + (1-c^2)}\bigg{)}. 
\end{equation} 
\noindent Since $\theta_u$ will be the same for all points on the line $r e^{i\theta_u}$, and since $c^2 -1 < 0$, we can assume without loss of generality that $r^2 = (1-c^2)$. Thus substituting into equation (\ref{eq:thetau}) we have: 
\begin{equation}\label{eq:thetanice}
\theta_u = \arctan\bigg{(}\frac{ \sin(2 \theta)}{ \cos(2\theta) + 1}\bigg{)} = \arctan\bigg{(}\frac{ 2\sin{\theta}\cos{\theta}}{ \cos^2(\theta)-\sin^2(\theta) + 1}\bigg{)} = \theta.
\end{equation} 
\noindent The exact same calculation shows that $\eta_s$ has a negative imaginary part.

We have shown that there can be no homoclinic orbit on the chart of $\C P^1$ parametrized by $\frac{w_2}{w_1}$. To see 
that on there can be no such orbit on the other chart, we have that off the origin, the charts are transformed into 
each other via $\eta \to \frac{1}{\eta}$ for $\zeta \in \C \backslash 0 $. Thus $\eta_u \to \frac{1}{\eta_u}$ which will have a 
negative imaginary part. Likewise $\frac{1}{\eta_s}$ will have a positive imaginary part. If we write $\zeta = \frac{\psi}{\psi'} = \vartheta + i \varsigma$ and consider the flow on this chart induced by (\ref{eq:mainode})
\begin{equation}\label{eq:odecp2}
\zeta ' = \frac{(\psi')^2-\psi \psi''}{\psi'^2} = 1 - (\frac{\cos v - \lambda^2}{c^2-1} + \frac{c^2\lambda^2}{(c^2-1)^2}) \zeta^2.
\end{equation}
\noindent Now letting $\zeta = \sigma + i \tau$ and considering only the imaginary part of the flow on $\C$ given by (\ref{eq:odecp2}) restricted to where $\tau = 0$ we have that:

\begin{equation}\label{eq:odeim2}
\tau' \big{|}_{\tau = 0} = \frac{-2pqc^2\sigma^2}{(c^2-1)^2} - \frac{2pq\sigma^2}{c^2-1} = -\frac{2pq \sigma^2}{(c^2-1)^2} < 0 \hspace{1in} \textrm{ if } p, q >0. 
\end{equation}
Thus the flow of (\ref{eq:odecp1}) on the real axis of (this chart of) $\C P^1$ is always pointing in the negative imaginary direction, and so there can be no heteroclinic orbit connecting $\frac{1}{\eta_u}$ to $\frac{1}{\eta_s}$, and hence there are no eigenvalues $\lambda$ with positive real part to equation (\ref{eq:lin3}). 

The same argument can be run in for $\lambda = p + iq$, with $p$ and $q$ both negative to show that the flow on the real axis is pointing in the wrong direction to allow for an orbit joining $\eta^u$ to $\eta^s$ as well. Just as in the previous case, one makes a Liouville transformation and then tracks the location of the stable and unstable subspaces on both charts of $\C P^1$. Thus we conclude that there are no eigenvalues off the imaginary axis. 

\subsection{Superluminal kink-waves}

We now return our attention to the superluminal eigenvalue problem  equation (\ref{eq:lin3}):

\begin{equation}\label{eq:lin4} 
 \begin{matrix}   \psi \\ (\cos v) \varphi - (c^2-1)\varphi_{zz} - 2c\psi_z \end{matrix}  = \lambda \begin{pmatrix} \varphi \\ \psi \end{pmatrix} 
 \end{equation}

We are interested in the values of $\lambda \in \C$ for which there are solutions to (\ref{eq:lin4}) which decay to zero as $z \to \pm \infty$. We 
first investigate the limiting cases $\lim_{z \to \pm \infty}$, which becomes the constant coefficient ODE 
 \begin{equation}\label{eq:lin5}
 (1-c^2) \varphi'' - 2c\lambda \varphi - (\lambda^2+1) \varphi = 0
 \end{equation} 
\noindent which can easily be solved for any value of $\lambda \in \C$. The characteristic exponents $r_{1,2}$ are given by: 
\begin{equation}\label{eq:charexp}
r_{1,2} = \frac{c\lambda \pm \sqrt{\lambda^2+(1-c^2)}}{1-c^2}
\end{equation}
\noindent By taking limits as $c\to 1^{+}$, we obtain that the signs of the real parts of $r_{1,2}$ are equal. In fact they will be the opposite of the sign of the real part of $\lambda$. We have: 
$$ \textrm{sgn } \left( \Re (\lim_{c\to 1^+} r_1) \right) = -\textrm{sgn }(\Re(2\lambda)) \qquad \textrm{ and, } 
\textrm{sgn } \left( \Re (\lim_{c\to 1^+} r_2) \right) = -\textrm{sgn }(\Re(\lambda+ \frac{1}{\lambda}))$$

Moreover, by differentiating each in the $c$ variable, we obtain that the sign of the derivative of the characteristic exponents $\frac{\partial r_{1,2}}{\partial c}$ is the 
same as the sign of the real part of $\lambda$. Thus we can conclude that in the case of $c^2-1>0$, we have that for the asymptotic cases, the 
eigenvalues of $A(\lambda)$ have the same sign. This means that as $z \to -\infty$ there can be no unstable orbit of the construction in lemma \ref{lem:geom} and therefore there can be no heteroclinic connection to the stable orbit at $\infty$. 
Thus we conclude that there is no point spectrum when $c^2-1>0$ is positive. This concludes the proof of theorem \ref{th:main}

\bibliographystyle{amsalpha}
\bibliography{sgbiblio}

\end{document}